\documentclass[oneside,english,11pt]{amsart}
\pdfoutput=1
\usepackage[T1]{fontenc}
\usepackage[latin9]{inputenc}
\usepackage{float}
\usepackage{units}
\usepackage{url}
\usepackage{amstext}
\usepackage{amsthm}
\usepackage{amssymb}
\usepackage{graphicx}

\makeatletter

\newcommand{\noun}[1]{\textsc{#1}}
\floatstyle{ruled}
\newfloat{algorithm}{tbp}{loa}
\providecommand{\algorithmname}{Algorithm}
\floatname{algorithm}{\protect\algorithmname}

\numberwithin{equation}{section}
\numberwithin{figure}{section}
\theoremstyle{plain}
\newtheorem{thm}{\protect\theoremname}
\theoremstyle{plain}
\newtheorem{prop}[thm]{\protect\propositionname}
\theoremstyle{remark}
\newtheorem{rem}[thm]{\protect\remarkname}
\theoremstyle{plain}
\newtheorem{cor}[thm]{\protect\corollaryname}

\usepackage{graphicx}
\makeatother

\usepackage{babel}
\providecommand{\corollaryname}{Corollary}
\providecommand{\propositionname}{Proposition}
\providecommand{\remarkname}{Remark}
\providecommand{\theoremname}{Theorem}

\begin{document}

\title{Choosing 1 of $N$ with and without lucky numbers}

\author{Matt Brand}

\address{MERL, 201 Broadway, Cambridge MA USA}


\date{2018.04.16 printed \today}
\begin{abstract}
How many fair coin tosses to choose 1 of $n$ options with uniform
probability? Although a probability problem, the solution is essentially
number-theoretic, with special roles for Mersenne numbers, Fermat
numbers, and the haupt exponent. We propose a bit-efficient scheme,
prove optimality, derive the expected number of coin tosses $e[n]$,
characterize its fractal structure, and develop sharp upper and lower
bounds, both discrete and continuous. A minor but noteworthy corollary,
with real-world examples, is that any lottery or simulation with finite
budget of random bits will have a predictable pattern of lucky and
unlucky numbers.
\end{abstract}

\maketitle

\section{Introduction}

How to choose fairly from 1 of $n$ alternatives, using a minimal
number of coin tosses? The problem is of practical interest for lotteries,
Monte Carlo simulations, and splitting a collegial lunch bill.\footnote{Choosing a payer at random has the virtues of being a fair policy
that requires no calculation and no memory of previous bills, payers,
or participants. Furthermore, it incentivizes participation in many
varied lunch cohorts to minimize lifetime deviation from the expected
cost of one's own eating. }

The 3-way case has some notoriety as a way of arbitrating disputes
in sports \cite{Bissinger90} and parenting. Popular schemes including
flipping three coins and choosing the odd man out; waiting for specific
patterns of heads or tails in a stream of coin flips; and a partition-of-unity
scheme that treats the sequence as the binary digits of a fraction.
Most schemes are inefficient; some are not correct. Indeed, the preliminary
steps of our analysis reveal an elementary error in random number
generation that can be demonstrated in many computing environments.

Formal treatments of the problem are rooted in Von Neumann's \cite{vonNeumann51}
procedure for obtaining an unbiased random bit from a coin of unknown
bias, which was subsequently generalized by Dwass \cite{Dwass72},
Bernard and Letac \cite{BernardLetac73} to choose 1 of $n$ outcomes
uniformly. These ``memoryless'' methods wait for specific sequences
of coinflip outcomes, and are thus inefficient. Stout and Warren \cite{StoutWarren84}
analyzed the tree of possible toss sequences to show that $O(\log n)$-toss
schemes exist, but did not achieve optimality. Using a similar strategy,
Kozen \cite{Kozen14} developed an optimal procedure for simulating
a $q$-biased coin with a $p$-biased coin.

This note develops a fair and optimally toss-efficient scheme for
choosing 1 of $n$. The scheme is easy to explain and follow, yet
it has rich mathematical structure: Its analysis revolves around factorizations
of Mersenne numbers $M_{k}\doteq2^{k}-1$ and ordinary Fermat numbers
$F_{k}\doteq2^{k}+1$, probability recurrences, residue systems, and
a fractal curve. The problem is framed in this section; $\mathsection$\ref{sec:InefficientSchemes}
analyzes some inefficient schemes and exposes a widespread flaw in
random number generation; $\mathsection$\ref{sec:EfficientScheme}
introduces a fair scheme, calculates its efficiency, and proves optimality.
The efficiency curve has a fractal structure (fig.~\ref{fig:expectationCurve})
with irregular peaks whose locations ($\mathsection$\ref{sec:Peak-locations}),
values ($\mathsection$\ref{sec:InefficientSchemes}), and upper bounds
($\mathsection$\ref{sec:Peak-values}) suggest a sub-logarithmic
property, which is proven in $\mathsection$\ref{sec:Orderings},
yielding a fair and efficient method for fair random orderings (equivalently,
samplings without replacement). Finally, appendix $\mathsection$\ref{sec:LuckyNumbers}
demonstrates how to find (or suppress) lucky numbers in lotteries
and some popular scientific computing packages. 

First, some basic facts: The lower bound on the number of fair tosses
is trivially $m\doteq\lceil\log_{2}n\rceil$ needed to distinguish
at least $n$ alternatives, and
\begin{prop}
\label{prop:noUpperBound}For $n$ with odd factors, a fair scheme
has no upper bound on the number of flips.
\end{prop}

\begin{proof}
By contradiction: Assume there exists a fair scheme that terminates
in no more than $T(n)<\infty$ tosses. Regardless of any early stopping
conditions, the scheme can be understood to assign all $2^{T(n)}$
outcomes to the $n$ choices. But there is no equal $n$-way partition
of $2^{k}$ objects for $n$ with odd factors, so the scheme cannot
be fair.
\end{proof}
It follows that a fair scheme must assign some number of outcomes
to a ``no decision'' condition that requires another round of tosses.
Consequently the ultimate number $t$ of coin tosses to choose $1$
of $n$ is a random variable, and we are interested in its expectation,
$e[n]\doteq\mathbb{E}_{n}[t]$.

\section{Some inefficient schemes\label{sec:InefficientSchemes}}

The expected number of coin tosses $e[n]$ will usually be determined
recursively. Two popular schemes are briefy considered to motivate
the analysis. W.l.og., we need only consider odd $n$, since $e[2n]=1+e[n]$. 

\subsection{Odd man out}

This scheme has some fame as a tie-breaker in sports and movies. It
affords a particularly simple analysis: There are $m=n$ tosses in
a round and $2n$ decision conditions wherein one toss is distinct
(odd). Thus the probability of continuing another round is $p=1-2n/2^{m}$.
Each round is independent, so the expectation is unchanged in next
round, giving us the recursion $e[n]=m+pe[n]$. Putting it together,
we have $e[n]=m+(2^{m}-2n)2^{-m}e[n]=2^{n-1}$ expected coin flips
for odd-man-out. This is decidedly inefficient; even $e[3]=4$.

\subsection{Partitions of unity}

This scheme is of interest because it is sometimes mistaken for optimal
and its analysis sheds light on a randomness deficiency that is ubiquitous
in deployed software. The unit interval $[0,1)\subset\mathbb{R}$
is divided into $n$ equal partitions $\{[0,\frac{1}{n}),[\frac{1}{n},\frac{2}{n}),\cdots,[\frac{n-1}{n},1)\}$
with $(n-1)$ interior boundaries at $\{\frac{1}{n},\frac{2}{n},\cdots,\frac{n-1}{n}\}$.
Then $m\ge\lceil\log_{2}n\rceil$ coin tosses select one of $2^{m}$
equal intervals $\{[0,\frac{1}{2^{m}}),[\frac{1}{2^{m}},\frac{2}{2^{m}}),\cdots,[\frac{2^{m}-1}{2^{m}},1)\}$.
Of these, $n-1$ are nondecisive because they span the aforementioned
boundaries. Landing in one requires at least one more coin toss, which
selects its decisive or nondecisive half. Therefore the expected number
of tosses needed for a decision is (nonrecursively)
\begin{equation}
e[n]=m+\frac{n-1}{2^{m-1}}(1+\nicefrac{1}{2}(1+\nicefrac{1}{2}(1+\cdots=m+\frac{n-1}{2^{m}}=\lceil\log_{2}n\rceil+\frac{n-1}{2^{\lceil\log_{2}n\rceil-1}}\,.\label{eq:partitionExpectation}
\end{equation}

Although more efficient and elegant\footnote{A nice property of the $n=3$ partition scheme is that boundaries
have binary representation $\nicefrac{1}{3}=.\overline{01}$ and $\nicefrac{2}{3}=.\overline{10}$,
therefore we have a fair decision as soon as two consecutive tosses
have the same outcome. To decode, represent each toss outcome as 0
or 1 and add together the first and last outcomes to get 0 or 1 or
2. } than odd-man-out, at $e[3]=3$ this method is still suboptimal.

\subsection{A sidenote on lucky numbers}

Between eqn.~\ref{eq:partitionExpectation} and prop.~(\ref{prop:noUpperBound}),
we have the curious implication that the two most common computer
implementations random integer function, \texttt{${\tt randomInteger}(n):=\lfloor n\cdot{\tt randomFloat}[0,1)\rfloor$}
and \texttt{${\tt randomInteger}(n):={\tt randomBinary}[0,2^{b})\bmod n$},
are neither efficient nor fair. More generally,
\begin{prop}
For $n\in\mathbb{N}$ with at least one odd factor, any size-$n$
lottery that uses a finite budget of coin tosses to choose a winning
number will have a predictable sequence of lucky numbers that are
up to twice as likely as unlucky numbers.
\end{prop}

\begin{proof}
W.l.o.g.~the lottery mechanism can be split into a random phase where
$b$ coin tosses choose one of $2^{b}$ possible outcomes (sequences)
and a deterministic phase where each possible outcome is surjectively
assigned to one of $n=2^{a}(2k+1)$ numbers for some integers $a\ge0,k\ge1$.
Since $n\nmid2^{b}$ and $n<2^{b}$, by the pigeonhole principle,
some numbers must be assigned at least one more outcome than the others.
If $n>2^{b-1}$, the ratio must be 2:1 exactly. Predictability follows
from the determinism of the assignment. 
\end{proof}
For $n\ll2^{b}$, the bias is negligible; but at lottery scale, lucky
numbers are mathematically plausible. Appendix~\ref{sec:LuckyNumbers}
lists some widely used computing platforms which currently exhibit
this bias, and illustrates how the pattern of lucky numbers can be
predicted from $n$. 

\section{A fair and efficient scheme\label{sec:EfficientScheme}}

For a fair and efficient scheme, we must use every bit of randomness
in both the decision and the no-decision sequences: Starting with
$m=\lceil\log_{2}n\rceil$ tosses, assign $n$ of the $2^{m}$ possible
outcome sequences to the $n$ choices; if any of these occur, we are
done. The remaining $r=2^{m}-n$ sequences will cue various \textquotedbl re-do\textquotedbl{}
scenarios, with collective probability $p=r/2^{m}$. E.g., if $r=1$
we start over with another round of $m$ tosses. If $n$ is Mersenne
number $M_{m}\doteq2^{m}-1$ this is always the case, so the expected
number of coin tosses follows the simple recursive invariant $e[n]=m+pe=m+2^{-m}e[n]$,
which solves to $e[M_{m}]=m2^{m}/(2^{m}-1)$. Thus choosing $1$ of
$3$ this way will cost only $e[3]=e[M_{2}]=\nicefrac{8}{3}$ tosses,
on average.

When $r\ge2$, only $m^{\prime}\doteq\lceil\log_{2}\nicefrac{n}{r}\rceil$
more coin tosses are needed to have $o\doteq r2^{m^{\prime}}>n$ equiprobable
outcomes in the next round. Assign those to $n$ choices and $r'\doteq r2^{m^{\prime}}-n$
re-do scenarios, and continue. Since the remainder $1<r'<n$ determines
the number of coin tosses in the next round, less than $n$ of the
potentially infinite rounds are unique, which means they must cycle,
and therefore there is always be a length $k<n$ recursive equation
for the expected number of coin tosses, of the form 
\begin{equation}
e[n]=m_{1}+r_{1}2^{-m_{1}}(m_{2}+r_{2}2^{-m_{2}}(m_{3}+r_{3}2^{-m_{3}}(\cdots+r_{k}2^{-m_{k}}e[n])\cdots)\label{eq:roundRecursion}
\end{equation}
where $r_{i}=(r_{i-1}2^{m_{i}})\bmod n$ and $m_{i}=\lceil\log_{2}\nicefrac{n}{r_{i-1}}\rceil$,
$r_{0}=1$. In $\mathsection$\ref{sec:algorithms}, algorithm~(\ref{alg:recursiveSolution})
solves this from the inside out using strictly integer arithmetic.

Since the cycle repeats at $r=1,$ the number of tosses in a cycle
of rounds will be any $T>1$ such that $n\mid M_{T}$, i.e., the discrete
logarithm of $1\bmod n\text{ base }2$. The smallest such $T$ is
variously known as the haupt exponent, , Sloane sequence \url{http://oeis.org/A002326},
and the multiplicative order of $2\pmod{n}$. We will write this as
$T=\text{ord}_{n}(2)$.

Given $T$, $e[n]$ can be expressed in terms of the residue class
$c_{i}\doteq2^{i}\bmod n$ as
\begin{prop}
The expected number of tosses to choose 1 of $n$ is
\begin{equation}
e[n]=\frac{2^{T}}{M_{T}}\sum_{i=0}^{T-1}\frac{2^{i}\bmod n}{2^{i}}\text{ for }T=k\cdot\mathrm{ord}_{n}(2)\text{ with any }k\in\mathbb{N}.\label{eq:expectationFromResidues}
\end{equation}
\end{prop}

\begin{proof}
Let $R$ be the number of rounds in a cycle. The residue sequence
$\{c_{i}\}_{i=0\ldots T-1}$ is exactly the concatenation of subsequences
$\{r_{t}2^{0},r_{t}2^{1},\cdots,r_{t}2^{m_{t+1}-1}\}_{t=0\ldots R-1}$
with remainders $r_{t}$ and toss-counts $m_{t}$ as defined above.
This contains the same information as eqn.~\ref{eq:roundRecursion},
but with rounds expanded into individual tosses. Putting this into
correspondence eqn.~\ref{eq:roundRecursion} yields expanded fixpoint
\begin{equation}
e[n]=1+p_{1}(1+p_{2}(1+p_{3}(1+\cdots+p_{T-1}(1+p_{T}e[n])\cdots)))\,.\label{eq:tossRecursion}
\end{equation}
with $p_{i}\doteq c_{i}/(2c_{i-1})$ being the probability of continuing
after the $i^{\text{th}}$ toss. Solving from the inside out yields
eqn.~\ref{eq:expectationFromResidues}. 
\end{proof}
\begin{rem}
$T$ can be any whole multiple of $\text{ord}_{n}(2)$ because the
fixpoint can be rolled out to any number of whole cycles. Some examples:
By Euler's theorem, the totient function $\phi(n)$ satisfies $n\mid M_{\phi(n)}$
for odd $n$, allowing $T=\phi(n)$. Also, for the Mersenne numbers
$n=M_{m}$ we have $T=m=\lceil\log_{2}n\rceil$. Finally, for prime
$n$, Lagrange's theorem states that the order of the generator ($=2$)
in the cyclic group $(\mathbb{Z}/n\mathbb{Z})^{\times}$ will divide
the order of the group, allowing $T=n-1$.
\end{rem}

The $e[n]$ recursion in eqn.~\ref{eq:tossRecursion} also highlights
the fact that the probability of no decision in $t$ flips is 
\begin{equation}
\text{Pr}(d>t)\doteq\prod_{i=0}^{t}p_{i}=c_{t}/(c_{0}2^{t})=2^{-t}(2^{t}\bmod n)\label{eq:no-decision}
\end{equation}
 and therefore the probability of terminating at the $t^{\text{th}}$
flip is 
\begin{equation}
Pr(t)\doteq\text{Pr}(d>t-1)-\text{Pr}(d>t)=2^{-t}(2c_{t-1}-c_{t})\:.\label{eq:prob-terminating}
\end{equation}
Note that if we fairly partition the space of $t$-flip sequences
into $n$ groups, there must be $2^{t}\bmod n$ leftover sequences
and $\text{Pr}(d>t)$ is exactly their probability mass. This is useful
for proving
\begin{prop}
\label{prop:Optimality}The bit-efficient scheme is optimal.
\end{prop}

\begin{proof}
By contradiction. Write the expectation as 
\[
e[n]=\sum_{t=1}^{\infty}t\,\text{Pr}(t)=\sum_{t=1}^{\infty}t(\text{Pr}(d>t-1)-\text{Pr}(d>t))=\sum_{t=0}^{\infty}\text{Pr}(d>t)\,,
\]
which is eqn.~\ref{eq:expectationFromResidues} with $T\to\infty$.
Suppose there is an alternate fair procedure with a lower expectation.
Then for some $t$, it must provide a smaller $\text{Pr}(d>t)$. However,
this cannot be achieved without assigning one or more leftover sequences
to a decision, hence the alternate procedure is unfair. 
\end{proof}
\begin{rem}
A visual proof can be had by growing a binary tree of all possible
coin toss sequences. At each level where $B\ge N$ branches appear,
we prune (assign) all but $B\bmod N$. If there were a more efficient
procedure, it would have to prune (assign) one of those branches earlier,
but that would make the assignee twice as likely as any of the other
decisions. Since this holds for any branching factor,
\end{rem}

\begin{cor}
The bit-efficient scheme generalizes to give an efficient procedure
for choosing 1 of $N$ with any $K>2$ sided fair die.
\end{cor}

\section{Peak locations\label{sec:Peak-locations}}

\begin{figure}
\noindent \centering{}\includegraphics[width=1\textwidth]{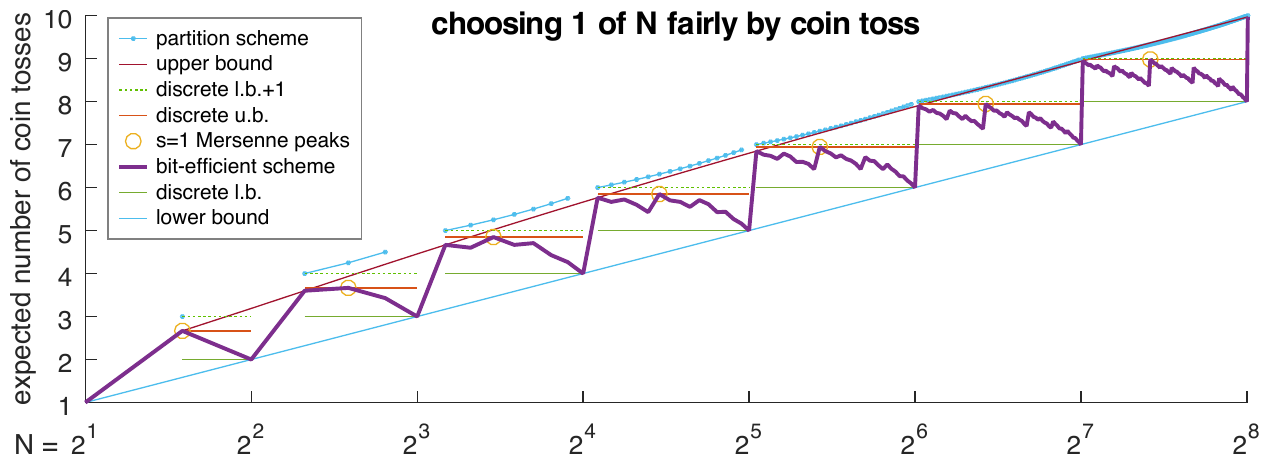}\caption{\label{fig:expectationCurve}$e[n]$ and bounding curves.}
\end{figure}
Viewed as a curve (see fig.~\ref{fig:expectationCurve}), the $e[n]$
sequence exhibits a fractal sawtooth structure with $m$ peaks in
the $m^{\text{th}}$ epoch ($2^{m-1}<n\le2^{m}$) that recur with
additional detail in subsequent epochs. Since eqn.~(\ref{eq:expectationFromResidues})
is an exponentially weighted sum of residues, the $e[n]$ curve spikes
whenever an increment in $n$ causes an early element of the residue
sequence $c_{i}=2^{i}\bmod n$ to jump in value. I.e., at $n=11,$
residue 5 jumps from $2\equiv2^{5}\pmod{10}$ to $10\equiv2^{5}\pmod{11}$.
The first $m$ residues are constant within each epoch; the remaining
residues determine the jumps.

In this section we determine the locations of these jumps; in the
next we determine the values of the associated peaks. Before giving
the general formula, we give special attention to the first two series
of peaks, which will later yield upper bounds on the whole sequence.
The first series occurs at the beginning of the epoch:
\begin{prop}[Fermat peaks]
Residue term $c_{k}$ makes a maximal jump at Fermat number $n=F_{k}$,
producing a peak.
\end{prop}

\begin{proof}
With $n=F_{k},$ the residue sequence $2^{i}\bmod n$ is $\{1,2,4,\cdots,2^{k-1},2^{k}=n-1,n-2,n-4,\cdots,n-2^{k-1}\}$
repeating, while the residue sequence for $2^{i}\bmod(n-1)$ is $\{1,2,4,\cdots,2^{k-1},0,0,0\cdots\}$
nonrepeating. Therefore the jump is $2^{k}-0=n-1$, the largest possible.
\end{proof}
The second series is associated with the second varying term in the
residue sequence. Its evolution is most conveniently tracked every
other epoch:
\begin{prop}[First Mersenne peaks]
In epoch $m=2k$, residue term $c_{2k+1}$ makes the second largest
jump at $n=F_{2k+1}/3$.
\end{prop}

\begin{proof}
We show that $2\equiv2^{2k+1}\pmod{n-1}$ and $n-1\equiv2^{2k+1}\pmod{n}$.
Useful facts are A: $3\mid M_{2k},\,k>0$ because $M_{2k}=M_{k}F_{k}=M_{k}(M_{k}+2)$;
B: $1\equiv a\pmod{b}$ and $c|b$ $\implies$$1\equiv a\pmod{b/c}$;
C: $1\equiv a\pmod{b}$$\implies$$2\equiv2a\pmod{2b}$; D: $n-1=(F_{2k+1}-3)/3=(2M_{2k}-2)/3=\nicefrac{2}{3}M_{2k}$;
E: $3\mid F_{2k+1},k\ge0$ because $3\mid M_{2k}$ and $F_{2k+1}=2M_{2k}+3$.
First, $1\equiv(M_{2k}+1)\pmod{M_{2k}}$ $\stackrel{A,B}{\implies}$
$1\equiv2^{2k}\pmod{M_{2k}/3}$ $\stackrel{C}{\implies}$ $2\equiv2\cdot2^{2k}\pmod{2M_{2k}/3}\stackrel{D}{=}2^{2k+1}\pmod{n-1}$.
Second, $n-1\equiv3n-1\pmod{n}$ $\implies$ $n-1\equiv2^{2k+1}\pmod{n}$
because $3n-1=3(F_{2k+1}/3)-1\stackrel{\text{E}}{=}2^{2k+1}$. Finally,
note that the jump $n-3$ is second largest because $n-1$ is the
maximal residue value and $2$ is the minimal value that can occur
in a (repeating) residue class after the first position.
\end{proof}
The remaining series of peaks are associated with smaller jumps:
\begin{prop}[All Mersenne peaks]
\label{prop:MersennePeaks}Residue $c_{m+s}$ jumps $n-F_{s}$ to
peak value $n-1$ in epoch $m=2sk$ at
\[
n=N_{s}[k]\doteq F_{s(2k+1)}/F_{s}
\]
\end{prop}

\begin{proof}
We claim that at $n=F_{s(2k+1)}/F_{s}\in\mathbb{N}$, $2^{s}\equiv2^{s(2k+1)}\pmod{n-1}$,
and $n-1\equiv2^{s(2k+1)}\pmod{n}$, yielding a jump of $n-1-2^{s}=n-F_{s}$.
Useful facts are A: $F_{s}\mid F_{s(2k+1)}$ because $(a-b)\mid(a^{q}-b^{q})$,
choose $a=2^{s},b=-1,q=2k+1$; B $M_{2sk}=(2^{2sk}-1)=(2^{sk}-1)(2^{sk}+1)=M_{sk}F_{sk}$;
C: $1\equiv a\pmod{b}$$\implies$$2^{s}\equiv2^{s}a\pmod{2^{s}b}$;
D: $n-1=(F_{s(2k+1)}-F_{s})/F_{s}=2^{s}(M_{ks}F_{ks})$. Fact A settles
claim 1, that $n\in\mathbb{N}$. For claim 2, $1\equiv(M_{2sk}+1)\pmod{M_{2sk}}$
$\stackrel{B}{\implies}$ $1\equiv2^{2sk}\pmod{M_{sk}F_{sk}}$ $\stackrel{C}{\implies}$
$2^{s}\equiv2^{s}\cdot2^{2sk}\pmod{2^{s}\cdot M_{sk}F_{sk}}\stackrel{D}{=}2^{s(2k+1)}\pmod{n-1}$.
For claim 3, $n-1\equiv F_{s}n-1\pmod{n}$ $\implies$ $n-1\equiv2^{s(2k+1)}\pmod{n}$
because $F_{s}n-1=F_{s}(F_{s(2k+1)}/F_{s})-1\stackrel{\text{A}}{=}2^{s(2k+1)}$. 
\end{proof}
To locate the peaks in the epochs between $2sk$ and $2s(k+1)$, it
is convenient to note that because the moduli $n$, $2n$, and $(2n-1)$
all generate similar residue sequences for the generator $2^{i}$,
each peak is echoed in subsequent epochs. In particular, from $N_{s}[k]$,
the $s^{\text{th}}$ peak propagates up through the epochs via $s$
iterations of $n\to2n$ alternating with $s$ iterations of $n\to2n-1$.
E.g., the $s=2$ peaks occur at $n={\bf 1},{\scriptstyle \times2=}2,{\scriptstyle \times2=}4,{\scriptstyle \times2-1=}7,{\scriptstyle \times2-1=}{\bf 13},26,52,103,{\bf 205},410,820,\cdots$
with the bolded entries being $N_{2}[k]$ for $k=0,1,2,\cdots$. We
call these the Mersenne peaks because the peak at $n=N_{s}(1)$ in
epoch $s$ is anticipated by a bump at Mersenne number $n=M_{s}$
in epoch $s-1$ under the progression $n\to2n-1$, i.e., $N_{s}(1)=2M_{s}-1$.

\section{Peak values\label{sec:Peak-values}}

By unrolling the recursive fixpoint at Mersenne peaks $n=N_{s}[k]$
and solving recurrences it is possible to get a closed-form solution
for selected $e[n]$. Defining $J_{s}[j]\doteq sj2^{sj}/F_{sj}$,
the $k^{\text{th}}$ instance of the $s^{\text{th}}$ peak has location
and value 
\begin{equation}
\begin{array}{rcl}
e\left[n=N_{s}[k]=F_{s(2k+1)}/F_{s}\right] & = & J_{s}[2k+1]-J_{s}[1]+2/F_{s}\\
 & = & \frac{s(2k+1)2^{s(2k+1)}}{F_{s(2k+1)}}-\frac{s2^{s}}{F_{s}}+\frac{2}{F_{s}}\\
 & = & 2sk+\frac{s+2}{F_{s}}-\frac{s(2k+1)}{F_{s(2k+1)}}
\end{array}\label{eq:expectationAtPeaks}
\end{equation}

\section{Bounds on $e[n]$}

We begin by summarizing the results of this section: For $m=\lceil\log_{2}n\rceil$
and $j=m+1+(m\bmod2)$, we have the bounds
\[
\begin{array}{rrclclcl}
\text{discrete:} & m & \le & e[n] & \le & 2+e[F_{j}/3]+m-j & < & m+1\\
\text{continuous:} & \log_{2}n & \le & e[n] & \le & 2+\nicefrac{(n-1)}{n}\log_{2}(n-1) & < & 2+\log_{2}n\,,
\end{array}
\]
with all $\leq$ bounds sharp. Fig.~(\ref{fig:expectationCurve})
depicts all these relationships.

The lower bounds are trivial. We begin with the loose discrete upper
bound. Regarding eqn.~(\ref{eq:expectationAtPeaks}), note that the
first term, $2sk$, is the number of the epoch in which the peak appears.
The second term attains a maximum value of 1 at $s=1,$ and the third
term always reduces the sum. Consequently $\text{2sk+1}$ upper-bound
all peaks in the epoch or, equivalently, 
\begin{prop}
$e[n]<\lceil\log_{2}n\rceil+1$. 
\end{prop}

\subsection{Continuous upper bound}

The largest spikes in the $e[n]$ sequence occur at the Fermat numbers
$n=F_{m}$, because these are associated with the largest jumps in
the residue sequence, and these jumps occur earliest in the exponentially
decaying sum for $e[n]$ in eqn.~\ref{eq:expectationFromResidues}.
Consequently we can construct an upper bound for the whole curve from
the Fermat peak values:
\begin{prop}[Fermat bound]
\label{prop:FermatPeakBound} For all $n$ in the $m^{\text{th}}$
epoch, $e[n]\le2+\frac{n-1}{n}\log_{2}(n-1)<2+\log_{2}n$ with equality
at $n=F_{m}$.
\end{prop}

\begin{proof}
To obtain $e[n=F_{m}]$, we unroll the recurrence in algorithm~(\ref{alg:recursiveSolution})
and observe that for Fermat numbers, the sequence $z[m]=2ng/(f-h)$
has recurrence $z[m]=2z[m-1]+M_{m-1}$, $z[1]=4$, whose solution
$z[m]=m2^{m-1}+2^{m}+1$ implies that $g/(f+h)=z[m]/(2n)=2+m2^{m}/F_{m}=2+J_{1}[m]$
at $n=F_{m}$. Substituting in $m=\log_{2}(n-1)$ gives $e[n=F_{m}]=2+\nicefrac{1}{n}(n-1)\log_{2}(n-1)$.
To establish the upper bound, note that any other peak located at
$N_{s}[k]$ in the same epoch (satisfying $m=2sk$) has the loose
upper bound $e[N_{s}[k]]<2sk+1=m+1=\lfloor\log_{2}(n-1)\rfloor+2$
for any $n>F_{m}$ in the $m^{\text{th}}$ epoch. Thus we need to
show that $\lfloor\log_{2}(n-1)\rfloor+2<2+\frac{n-1}{n}\log_{2}(n-1)$
at any non-Fermat peak. In any epoch, the \noun{lhs} is constant and
\noun{rhs} is increasing, so it suffices to show that it holds at
the first non-Fermat peak, which occurs at $n=N_{1}[k]=F_{2k+1}/3,k>1$.
There we have $\frac{n-1}{n}=1-3/F_{2k+1}$; choose worst case $k=2$.
Then \noun{lhs} is 5 and \noun{rhs} is $2+(1-3/F_{5})\log_{2}(F_{5}/3-1)=2+(10/11)\log_{2}10>5$.
\end{proof}

\subsection{Discrete upper bound}

Since early residues have exponentially larger weights in the expectation,
peaks from the $s=1$ Mersenne peaks dominate all others. This  can
be expressed as a \textquotedbl stair-step\textquotedbl{} curve that
gives a constant (but sharp) upper bound for each epoch:
\begin{prop}[Mersenne bound]
Let $j=2k+1$ for $k\in\mathbb{N}$. Then $e[n]\le e[F_{j}/3]=j2^{j}/F_{j}$
for all $n$ in epoch $m=j-2$ and $e[n]\le1+e[F_{j}/3]$ for all
$n$ in epoch $m=j-1$. 
\begin{proof}
The odd $m$ case is simply the value $e[n]$ at peak $n=N_{1}(k)$
from eqn.~\ref{eq:expectationAtPeaks}, line 2. The even case follows
from $e[2n]=e[n]+1$ and the peak propagation rule. To show that this
also upper-bounds the Fermat peaks, we need $e[n]=2+\frac{n-1}{n}\log_{2}(n-1)<j2^{j}/F_{j}$
at Fermat peak locations $n=F_{j-2}$. Putting all in terms of $k=(j-1)/2$
and taking \noun{rhs-lhs} yields $\frac{4^{k}(6k-5)-4}{2F_{2k-1}F_{2k+1}}$,
which is zero at $k=1$ and positive for $k>1$.
\end{proof}
\end{prop}

\section{Orderings\label{sec:Orderings}}

We close with result on fair orderings: It is more toss-efficient
to determine an ordering of $k$ objects by choosing 1 of $k!$, than
by sequentially choosing 1 of $k$, then 1 of $k-1,$ etc. Specifically,
\begin{prop}[Logarithmic subadditivity]
$\forall_{a,b}\:e[ab]\le e[a]+e[b]$ . 
\end{prop}

\begin{proof}
We prove $\forall_{a,b}\ngtr$, $\exists_{a,b}=$, $\exists_{a,b}<$.
Claim 1: By prop.~, $e[n]$ is toss-efficient. No toss-efficient
procedure can have $\exists_{a,b}\:e[ab]>e[a]+e[b]$ because one could
then reduce $e[ab]$ by simply choosing 1 of $a$ and then 1 of $b$.
Claim 2: Since $e[2n]=1+e[n]$, we have equality $e[ab]=e[a]+e[b]$
when $a$ or $b$ is a power of 2. Claim 3: Consider a Fermat peak
at $a=F_{s}$ and a Mersenne peak at $b=F_{s(2k+1)}/F_{s}$ (see prop.~\ref{prop:MersennePeaks}).
Using the peak-value formulas from prop.~\ref{prop:FermatPeakBound}
and eqn.~\ref{eq:expectationAtPeaks} we obtain $e[a]+e[b]-e[ab]=e[F_{s}]+e[F_{s(2k+1)}/F_{s}]-e[F_{s(2k+1)}]=2/F_{s}$.
\end{proof}
We further conjecture that the inequality is strict whenever $a$
and $b$ both have odd factors.

\bibliographystyle{amsplain}
\bibliography{cointoss}

\appendix

\section{Lucky numbers and how to find them\label{sec:LuckyNumbers}}

The pattern of lucky numbers in a finite lottery will depend on the
mapping from random flips to ticket numbers. The common programming
idiom 
\[
\text{randomInteger}(n):=\text{randomInteger}(2^{b})\bmod n
\]
 assigns more random outcomes to lower ticket numbers, whereas 
\[
\text{randomInteger}(n):=\lfloor n\cdot{\tt randomFloat}[0,1)\rfloor\:,
\]
 widely enshrined in the libraries of scientific computing environments,
has patterned lucky numbers, exemplifed by
\begin{prop}
Given $b$, choose any $n=2^{a}M_{k}$ with $a\ge0$, $k>1$, $a+k\le b$,
and let $p\sim\mathcal{U}_{b}(1)$ be a random variable uniformly
sampled from the set $U_{b}\doteq\{0,\nicefrac{1}{2^{b}},\nicefrac{2}{2^{b}},\nicefrac{3}{2^{b}},\cdots,\nicefrac{2^{b}-2}{2^{b}},\nicefrac{2^{b}-1}{2^{b}}\}$.
Then random integers $q\doteq\lfloor np\rfloor\in\{0,\cdots,n-1\}$
and $q\bmod M_{k}$ are both nonuniformly distributed with a length-$M_{k}$
pattern of lucky and unlucky numbers. Furthermore, if $k\mid$$b-a$,
then given two numbers $0\le q_{L},q_{U}<n$ with $M_{k}-1\equiv q_{L}\pmod{M_{k}}$
and $M_{k}-1\not\equiv q_{U}\pmod{M_{k}}$, $q_{L}$ is lucky with
$\frac{Pr(q=q_{L})}{Pr(q=q_{U})}=$$\frac{c+1}{c}$ with $c=\lfloor\frac{2^{b-a}}{M_{k}}\rfloor\in\mathbb{N}$
as small as $c=1$ for $n=3\cdot2^{b-2}$.
\end{prop}

\begin{proof}
$\lfloor np\rfloor\pmod{M_{k}}$ wraps $U_{b}$, viewed as a knotted
string, around a cylinder of circumference $M_{k}$ with $M_{k}$
equidistant notches, then assigns each knot to the closest notch in
one direction. Since $2^{a}\mid2^{b}$, the string wraps exactly $2^{a}$
times such that each consecutive group of $2^{b-a}$ samples from
$U_{b}$ is identically assigned to elements of the residue set $\{0,\cdots,M_{k}-1\}$.
Since $2^{b-a}>M_{k}$ and $M_{k}\nmid2^{b-a}$, by the pigeonhole
principle, $2^{b-a}\bmod M_{k}$ elements of the residue set get assigned
one extra sample. If $k\mid b-a$, exactly $1\equiv2^{b-a}\pmod{M_{k}}$
element gets this surplus, because $(x-y)\mid(x^{z}-y^{z})$ and thus
$k\mid b-a$$\implies$($2^{k}-1)\mid(2^{k(b-a)/k}-1)$$\implies$$M_{k}\mid M_{b-a}$$\implies$$1\equiv2^{b-a}\pmod{M_{k}}$.
The lucky element's position in the residue set is: last for $\lfloor np\rfloor$;
first for $\lceil np\rceil;$ middle for $\lfloor np\rceil$. The
remainder of the proposition follows from arithmetic.
\end{proof}
In practice, the pattern and prominence of lucky numbers will also
depend on the number of bits used internally for CPU arithmetic and
the choice of rounding scheme. At time of writing the bias can be
demonstrated in several numerically sophisticated programming environments,
even when $M_{k}$ is replaced by an arbitrary odd number. For example,
with $b=53$ bits in the IEEE754 double-precision floating point significand,
let $j$ be a small integer, $f=F_{j}$, and $n=\lfloor2^{b-1}(F_{j}/F_{j-1})\rceil$.
Then in $k$ trials, lottery ticket numbers $\pmod{f}$ are winners
in Matlab or Octave with non-uniform relative frequencies
\[
\mathtt{full(sparse(1+mod(randi(n,k,1),f),1,1,f,1))/k}
\]

Similar bias can be elicited from the ${\tt sample}()$ and ${\tt runif}$()
functions in the R statistics environment. Python and Mathematica
use arbitrary-precision arithmetic which drives the bias down to insignificant
levels. 

Rejection sampling offers a trivial but inefficient way to restore
fairness: Obtain a fair $R\in\{1,\ldots,2^{b}\}$ from a round of
$b$ tosses, reject and repeat if $R>n$. Naïvely this takes a impractical
$e'=1+(1-n/2^{b})e'=2^{b}/n$ rounds of $b$ tosses each. But if we
instead reject when $R>n\lfloor2^{b}/n\rfloor$, the largest multiple
of $n$ that is $\le2^{b}$, then $R\bmod n$ is a fair draw from
$\{0,\ldots,n-1\}$ and the expectation falls to $e'=1/(1-(2^{b}\bmod n)/2^{b})$
rounds. In tosses, $b\le e<2b$, $e[3]=b$ and $e[F_{b-1}]=2b(1-1/F_{b})$.
Although quite bit-inefficient, C++ libraries appear to use an (algebraically)
equivalent strategy.

\section{\label{sec:algorithms}Solution for $e[n]$ recurrence in (\ref{eq:roundRecursion})}

\begin{algorithm}[H]
\caption{\label{alg:recursiveSolution}An integer-arithmetic solution of the
recursive fixpoint for $e[n]$. $m$ is the number of coin flips in
each round; $o$ is the number of possible outcomes; $r$ is the number
of unassigned outcomes; and $r/o$ is the probability of going another
round.}

\textbf{Input:} Odd $n\ge3$.
\begin{description}
\item [{Initialize}] \noindent \begin{raggedright}
$f\leftarrow1,\:g\leftarrow0,\:h\leftarrow1,\:r\leftarrow1.$
\par\end{raggedright}
\item [{Iterate}] \noindent \begin{raggedright}
$m,o\leftarrow\min_{m}\text{s.t. }o=2^{m}r>n$,\\
\qquad{}\qquad{}~ $r\leftarrow o-n$,~ $f\leftarrow fo$,~ $g\leftarrow(g+mh)o$,~
$h\leftarrow hr$.
\par\end{raggedright}
\item [{Until}] \noindent \begin{raggedright}
$r=1$.
\par\end{raggedright}
\end{description}
\textbf{Output:} $e[n]=g/(f-h)$.
\end{algorithm}

\section*{}
\end{document}